\documentclass[11pt]{article}
\usepackage{graphicx}

\usepackage{amsmath,amsfonts,amssymb,amsthm}
\usepackage{bbm,mathrsfs, bm, mathtools}

\title{Poisson hyperplane processes\\ and approximation of convex bodies}
\author{Daniel Hug and Rolf Schneider}
\date{}
\newcommand{\Sd}{{\mathbb S}^{d-1}}
\newcommand{\R}{{\mathbb R}}

\newcommand{\K}{{\mathcal K}}

\newcommand{\bP}{{\mathbb P}}

\newcommand{\N}{{\mathbb N}}

\newcommand{\Ha}{\mathcal{H}}

\newcommand{\D}{{\rm d}}

\newcommand{\bE}{{\mathbb E}\,}

  \renewcommand{\exp}{{\rm exp}\,}

  \newcommand{\fed}{\,\rule{.1mm}{.20cm}\rule{.20cm}{.1mm}\,}

\newtheorem{theorem}{Theorem}
\newtheorem{lemma}{Lemma}

\begin{document}
\maketitle

\begin{abstract}
A natural model for the approximation of a convex body $K$ in $\R^d$ by random polytopes is obtained as follows. Take a stationary Poisson hyperplane process in the space, and consider the random polytope $Z_K$ defined as the intersection of all closed halfspaces containing $K$ that are bounded by hyperplanes of the process not intersecting $K$.
If $f$ is a functional on convex bodies, then for increasing intensities of the process, the expectation of the difference $f(Z_K)-f(K)$ may or may not converge to zero. If it does, then the order of convergence and possible limit relations are of interest. We study these questions if $f$ is either the hitting functional or the mean width.\\[1mm]
\end{abstract}

\section{Introduction}\label{sec1}

The approximation of convex bodies by random polytopes is a major theme in Stochastic Geometry. Each of the survey articles \cite{Bar07, Bar08}, \cite{Buc85}, \cite{Hug13}, \cite{Rei10}, \cite{Sch88, Sch18}, \cite{WW93}, and also Section 8.2 of the book \cite{SW08}, provides information about this topic. In this context, an often used model for a random polytope is the convex hull of $n$ independent random points with uniform distribution in a given convex body $K$ in $\R^d$. As $n$ tends to infinity, one is interested in the behavior of some functional evaluated at this convex hull, in comparison to the same functional evaluated at $K$. Replacing the set of $n$ independent random points by a Poisson process of intensity $n$ is  often an advantage, due to the strong independence properties of such processes. The dual generation of polytopes, by intersecting finitely many closed halfspaces, leads to new aspects. Combining this with Poisson processes, we are led to the following model. We consider a stationary Poisson hyperplane process $\widehat X$ in $\R^d$ with directional distribution $\varphi$ (an even finite Borel measure on the unit sphere) and intensity $n\in\N$. (This choice of intensity is inessential and is only made to stress the analogy to models using $n$ independent, identically distributed random points.) A convex body $K\subset\R^d$ is given, and we define the $K$-cell $Z_K^{(n)}$ of $\widehat X$ as the intersection of all closed halfspaces containing $K$ that are bounded by hyperplanes of $\widehat X$ not intersecting $K$. A special feature of this model consists in the fact that the $K$-cell $Z_K^{(n)}$ need not converge a.s. to $K$ as $n\to\infty$; whether it does, will depend on the directional distribution of the hyperplane process.

A first result on the asymptotic behavior of the $K$-cell was proved by Kaltenbach \cite{Kal90}, who considered the volume $V$. Under the assumption that the directional distribution of the hyperplane process $\widehat X$ has a positive, continuous density with respect to spherical Lebesgue measure, he proved that
\begin{equation}\label{1.1}
n^{-\frac{2}{d+1}}\ll \bE V(Z_K^{(n)})-V(K)\ll n^{-\frac{1}{d}}.
\end{equation}
Here $\bE$ denotes the expectation, and $f\ll g$ means that there exists a constant $c$, independent of $n$, such that $f(n)\le cg(n)$ for all sufficiently large $n$. The constant $c$ may depend on $d,\varphi,K$ (where $d$ is determined by $K$, if $K$ has interior points).

The estimates (\ref{1.1}) immediately provoke the question whether they can be extended from the volume $V=V_d$ to the general intrinsic volume $V_i$. For the first intrinsic volume $V_1$, which is proportional to the mean width $W$, we give here the following answer. By $\sigma$ we denote the normalized spherical Lebesgue measure on the unit sphere $\Sd$ of $\R^d$.

\begin{theorem}\label{T1.1}
Suppose that the stationary Poisson hyperplane process $\widehat X$ has intensity $n\in\N$ and spherical directional distribution $\varphi$. Let $K\subset\R^d$ be a convex body with interior points. Then the mean width $W$ of the $K$-cell $Z_K^{(n)}$ satisfies
\begin{equation}\label{1.2}
n^{-1}\log^{d-1}n  \ll \bE W(Z_K^{(n)}) -W(K) \ll n^{-\frac{2}{d+1}}
\end{equation}
under the following assumptions:

The left inequality holds if the directional distribution satisfies
\begin{equation}\label{1.3}
\varphi\le a_0\sigma
\end{equation}
with some constant $a_0$. The right inequality holds if
\begin{equation}\label{1.4}
\varphi\ge a_1\sigma
\end{equation}
with some constant $a_1>0$.
\end{theorem}

In the case where $\widehat X$ is also isotropic, that is, satisfies $\varphi=\sigma$ (equivalently, the distribution of $\widehat X$ is invariant under rotations), this was proved in \cite{Sch19}. The proof given there for the lower estimate can be extended to the non-isotropic case. The upper estimate, however, requires a different approach.

Under the assumption that the directional distribution $\varphi$ has a positive, continuous density with respect to spherical Lebesgue measure, we prove in Section \ref{sec3} (Theorem \ref{T3.1}) a precise asymptotic relation for $\bE W(Z_K^{(n)}) -W(K)$, which shows that the order on the right side of (\ref{1.2}) is attained by sufficiently smooth convex bodies. Also the order on the left side of (\ref{1.2}) is sharp, as we show by another asymptotic relation, holding for simplicial polytopes.

We remark that an analogue of (\ref{1.1}) or (\ref{1.2}) for the intrinsic volume $V_i$, $i\not=1,d$, remains an open problem.

The right-hand estimate of (\ref{1.2}) will be derived from more general results about the hitting functional, which are of independent interest. This requires a few explanations. The space of convex bodies (nonempty, compact, convex sets) in $\R^d$ is denoted by $\K$. It is equipped with the Hausdorff metric $d_H$. The space of hyperplanes in $\R^d$ with its usual topology is denoted by $\Ha$. Hyperplanes and closed halfspaces are often written in the form
$$ H(u,\tau)=\{x\in\R^d:\langle x,u\rangle=\tau\},\qquad  H^-(u,\tau)=\{x\in\R^d:\langle x,u\rangle \le\tau\}$$
with $u\in\Sd$ (the unit sphere of $\R^d$) and $\tau\in \R$, where $\langle\cdot\,,\cdot\rangle$ denotes the scalar product of $\R^d$. For a convex body $M\in\K$ we denote by
$$ \Ha_M:= \{H\in\Ha: H\cap M\not=\emptyset\}$$
the set of hyperplanes hitting $M$.

That $\widehat X$ is a Poisson hyperplane process in $\R^d$ means that $\widehat X$ is a mapping from some probability space $(\Omega,{\bf A},\bP)$ into the measurable space of simple, locally finite counting measures on $\Ha$ with the following properties: $\{\widehat X(A)=0\}$ is measurable for all Borel sets $A\subset \Ha$,  the intensity measure $\widehat\Theta =\bE \widehat X$ is locally finite, and we have
$$ \bP(\widehat X(A)=k) =e^{-\widehat\Theta(A)}\frac{\widehat\Theta(A)^k}{k!}$$
for $k\in\N_0$ and each Borel set $A\subset\Ha$ with $\widehat\Theta(A)<\infty$. (For more information, also about independence properties, we refer to \cite{SW08}, Sections 3.1 and 3.2.) As usual, we identify a simple counting measure with its support; for example, we write $H\in\widehat X$ for $\widehat X(\{H\})=1$. The hyperplane process $\widehat X$ is stationary if its intensity measure (and hence its distribution) is invariant under translations. In that case, one has a unique representation
$$ \widehat\Theta(A) =\widehat\gamma \int_{\Sd}\int_{-\infty}^\infty {\mathbbm 1}_A(H(u,\tau))\,\D\tau\,\varphi(\D u)$$
for Borel sets $A\subset\Ha$ (see, e.g., \cite[(4.33)]{SW08}). Here $\widehat \gamma$ is the {\em intensity} of $\widehat X$, later assumed to be a number $n\in\N$, and the even probability measure $\varphi$ is the {\em directional distribution} of $\widehat X$.

If now $K\subset\R^d$ is a convex body, then the expected number of hyperplanes of the process hitting $K$ is given by
$$ \bE\widehat X(\Ha_K) = 2\widehat\gamma\Phi(K)$$
with
$$ \Phi(K) := \int_{\Sd} h(K,u)\,\varphi(\D u),$$
where $h(K,\cdot)$ denotes the support function of $K$. Therefore, the functional $\Phi$ on convex bodies is called the {\em hitting functional}. It should be compared to the mean width, which is given by
$$ W(K) = 2\int_{\Sd} h(K,u)\,\sigma(\D u).$$
In particular, $2\Phi=W$ if $\widehat X$ is isotropic. The hitting functional has similar properties as the mean width functional: it is continuous with respect to the Hausdorff metric, translation invariant, homogeneous of degree one, and  (weakly) increasing under set inclusion. In the non-isotropic case, the hitting functional is better adapted to the process than the mean width, as shown by the following result.

\begin{theorem}\label{T1.2}
If $\widehat X$ has intensity $n$ and $K\in \K$ is a convex body with interior points, then
\begin{equation}\label{1.5}
\bE \Phi(Z_K^{(n)})-\Phi(K)\ll n^{-\frac{2}{d+1}}.
\end{equation}
\end{theorem}

It should be noted that there is no further assumption on $K$ and no assumption on $\varphi$. Under these general circumstances, the $K$-cell $Z_K^{(n)}$ need not converge a.s. to $K$ as $n\to \infty$ (for example, if $\varphi$ is discrete and $K$ is smooth); nevertheless, $\bE \Phi(Z_K^{(n)})$ converges to $\Phi(K)$.

In Section \ref{sec2}, we shall derive with Theorem \ref{T2.1} a deviation estimate for $\Phi(Z_K^{(n)})-\Phi(K)$. From it, a more general version of Theorem \ref{T1.2} is derived, namely Theorem \ref{T2.2}, which provides moment estimates for $\Phi(Z_K^{(n)})-\Phi(K)$. Theorem \ref{T1.1} is then proved in Section \ref{sec3}, where we also obtain precise asymptotic relations under additional assumptions. In Section \ref{sec4} we  deal with the facet number of the $K$-cell. We prove an estimate for its moments and, under special assumptions, two asymptotic relations.

\section{The hitting functional}\label{sec2}

The assumption from now on is that $\widehat X$ is a stationary Poisson hyperplane process in $\R^d$ with directional distribution $\varphi$ and intensity $n\in \N$, and that $K\in\K$ is a convex body with interior points. Since $\widehat X$ is stationary, we can assume without loss of generality that the origin $o$ of $\R^d$ is contained in the interior of $K$. We assume, in fact, that $o$ is the centroid of $K$ (in order that constants depending on $K$ and the position of the origin will depend only on $K$). The $K$-cell of $\widehat X$ is defined by
$$ Z_K^{(n)}:= \bigcap_{H\in\widehat X,\,H\cap K=\emptyset} H_K^-,$$
where for a hyperplane $H$ not intersecting $K$ we denote by $H_K^-$ the closed halfspace bounded by $H$ that contains $K$.

For a bounded subset $M\subset\R^d$ we denote by $R_o(M)$ the radius of the smallest ball with center $o$ that contains $M$. The radius $R_o(K)$ is abbreviated by $R_o$. We quote the following lemma from \cite{Sch19}.

\begin{lemma}\label{L2.1}
There are constants $a,b>0$, depending only on $\varphi$, such that
$$ \bP\left(R_o(Z_K^{(n)})>b(R_o+x) \right)\le 2de^{-anx} \quad\mbox{for }x\ge 0.$$
\end{lemma}

We set $R:= bR_o+1$ and $B_R:=RB^d$, where $B^d$ denotes the unit ball of $\R^d$. Then Lemma \ref{L2.1} shows that
\begin{equation}\label{2.1}
\bP\left(R_o(Z_K^{(n)})\ge R+t\right)\le 2de^{-cn(t+1)} \quad\mbox{for }t\ge 0
\end{equation}
(with $c=a/b$) and, therefore,
\begin{equation}\label{2.2}
\bP\left(Z_K^{(n)}\not\subset B_R\right)\le 2de^{-cn}.
\end{equation}
For this reason, we may restrict the subsequent estimate to the $K$-cells contained in $B_R$.

The following result about $\varepsilon$-nets in the space of convex bodies with the Hausdorff metric $d_H$ was proved by Bronshtein \cite{Bro76}.

\begin{lemma}\label{L2.2}
Let $\varepsilon>0$. There exist convex bodies $M_1,\dots,M_{N_\varepsilon}\subset B_R$, where
$$N_\varepsilon\le c_1\exp\varepsilon^{-\frac{d-1}{2}}$$
with a constant $c_1$ depending only on $d$ and $R$, such that to each convex body $M\subset B_R$ there exists a number $j\in\{1,\dots,N_\varepsilon\}$ with $d_H(M,M_j)\le\varepsilon$.
\end{lemma}

We have learned about the usefulness of this lemma for random approximation from a paper of Brunel \cite{Bru17}, where it is applied to convex hulls of random points. It is mentioned in \cite{Bru17} that this approach was inspired by Theorem 1 in \cite{KST95}. We give a dual and Poisson version of Brunel's approach. With the aid of Lemma \ref{L2.2}, we prove the following deviation estimate.

\begin{theorem}\label{T2.1}
There are constants $c(K), n_0(K)>0$, depending only on $K$, such that
$$ \bP\left(Z_K^{(n)}\subset B_R \,\wedge\,\Phi(Z_K^{(n)})-\Phi(K)\ge c(K)n^{-\frac{2}{d+1}}+x  \right)\le c_1e^{-nx} \quad\mbox{for }x\ge 0,$$
for all $n\in\N$ with $n\ge n_0(K)$.
\end{theorem}

\begin{proof}
For $M\in\K$ we have
\begin{equation}\label{2.3}
|\Phi(M)-\Phi(K)|\le \int_{\Sd} |h(M,u)-h(K,u)|\,\varphi(\D u) \le d_H(M,K).
\end{equation}

Since we have assumed that $o$ is the centroid of $K$, there is a number $r>0$, depending only on $K$, such that $2rB^d\subset {\rm int}\,K$.
Let $0<\varepsilon<r/2$ be given. With each $M_j$, $j\in\{1,\dots,N_\varepsilon\}$, we associate the convex body
$$ M_j^{-\varepsilon}:={\rm conv}\left(K\cup\left(1-\frac{\varepsilon}{r}\right)M_j\right).$$
Then $K\subseteq M_j^{-\varepsilon}$.

Under the condition that $Z_K^{(n)}\subset B_R$, by Lemma \ref{L2.2} there is a (random) number $k\in\{1,\dots,N_\varepsilon\}$ with $d_H(Z_K^{(n)},M_k)\le\varepsilon$. Then $rB^d\subset {\rm int}\,M_k$, and from (\ref{2.3}) it follows that
$$\Phi(M_k)-\Phi(K)\ge \Phi(Z_K^{(n)})-\Phi(K)-\varepsilon.$$
For $u\in\Sd$ we have $h(M_k^{-\varepsilon},u)=h(K,u)$ or $h(M_k^{-\varepsilon},u)=(1-\varepsilon/r)h(M_k,u)$. In the first case,
$$ h(M_k^{-\varepsilon},u)=h(K,u)< h(Z_K^{(n)},u)$$
almost surely. In the second case, because of $|h(Z_K^{(n)},u)-h(M_k,u)|\le \varepsilon$ and $h(M_k,u)> r$, we have
$$ h(M_k^{-\varepsilon},u) =\left(1-\frac{\varepsilon}{r}\right)h(M_k,u) < h(M_k,u)-\varepsilon\le h(Z_K^{(n)},u).$$
Therefore,
$$ M_k^{-\varepsilon}\subset{\rm int}\, Z_K^{(n)}$$
almost surely. It follows that a.s. no hyperplane of $\widehat X\setminus\Ha_K$ intersects $M_k^{-\varepsilon}$.

In the following we write
$$ I_t:= \{j\in \{1,\dots,N_\varepsilon\}:\Phi(M_j)-\Phi(K)>t\}$$
for $t\in\R$. Let $t\ge\varepsilon$ be given. We  obtain
\begin{eqnarray*}
&&\bP\left(Z_K^{(n)}\subset B_R \,\wedge\,\Phi(Z_K^{(n)})-\Phi(K)>t\right)\\
&&\le \bP\left(\exists\, j\in \{1,\dots,,N_\varepsilon\}:\Phi(M_j)-\Phi(K)>t-\varepsilon \mbox{ and } H\cap M_j^{-\varepsilon}=\emptyset\;\forall H\in\widehat X\setminus\Ha_K\right)\\
&& \le \bP\left(\bigcup_{j\in I_{t-\varepsilon}}\left\{H\cap M_j^{-\varepsilon}=\emptyset\;\forall H\in\widehat X\setminus\Ha_K\right\}\right)\\
&&\le \sum_{j\in I_{t-\varepsilon}}\bP\left(H\cap M_j^{-\varepsilon}=\emptyset\;\forall H\in\widehat X\setminus\Ha_K\right)\\
&&= \sum_{j\in I_{t-\varepsilon}} \exp\left[-\widehat\Theta\left(\Ha_{M_j^{-\varepsilon}} \setminus\Ha_K\right)\right]\\
&&= \sum_{j\in I_{t-\varepsilon}} \exp\left[-2n\left(\Phi\left(M_j^{-\varepsilon}\right)-\Phi(K)\right)\right].
\end{eqnarray*}
Here we have used that $K\subseteq M_j^{-\varepsilon}$ and hence $\Ha_ K\subseteq \Ha_{M_j^{-\varepsilon}}$. If $j\in I_{t-\varepsilon}$, we also have (using the translation invariance and monotonicity of $\Phi$)
$$ \Phi(M_j^{-\varepsilon})\ge \left(1-\frac{\varepsilon}{r}\right)\Phi(M_j)\ge \left(1-\frac{\varepsilon}{r}\right)(\Phi(K)+t-\varepsilon)$$
and hence, using that $\varepsilon <r/2$,
\begin{eqnarray*}
2n\left(\Phi\left(M_j^{-\varepsilon}\right)-\Phi(K)\right) &\ge & 2n\left[\left(1-\frac{\varepsilon}{r}\right)(t-\varepsilon)-\frac{\varepsilon}{r}\Phi(K)\right]\\
&\ge& n\left[t-\varepsilon-\frac{2\varepsilon}{r}\Phi(K)\right]\\
&=&n[t-c'(K)\varepsilon]
\end{eqnarray*}
with $c'(K):= 1+\frac{2}{r}\Phi(K)$. Therefore, using Lemma \ref{L2.2},
\begin{eqnarray*}
\bP\left(Z_K^{(n)}\subset B_R \,\wedge\,\Phi(Z_K^{(n)})-\Phi(K)>t\right)
&\le& N_\varepsilon\exp\left[-n(t-c'(K)\varepsilon)\right]\\
&\le& c_1\exp \varepsilon^{\frac{1-d}{2}}\exp\left[-n(t-c'(K)\varepsilon)\right]\\
& = &c_1\exp\left[-n(t-c'(K)\varepsilon)+\varepsilon^{\frac{1-d}{2}}\right].
\end{eqnarray*}
This holds for arbitrary $\varepsilon$ with $0<\varepsilon<r/2$ (where $r$ depends on $K$). Now we choose $\varepsilon= n^{-\frac{2}{d+1}}$ with sufficiently large $n$. Since $\varepsilon^{\frac{1-d}{2}}=n\varepsilon$, we obtain
$$ \bP\left(Z_K^{(n)}\subset B_R \,\wedge\,\Phi(Z_K^{(n)})-\Phi(K)>t\right)\le c_1\exp\left[-n(t-c(K)\varepsilon)\right]$$
with $c(K)=2(1+\Phi(K)/r)$.
The choice $t=c(K)n^{-\frac{2}{d+1}}+x$ yields the assertion.
\end{proof}

From this, we can derive upper estimates for the moments of the difference $\Phi(Z_K^{(n)})-\Phi(K)$.

\begin{theorem}\label{T2.2}
For $k\in\N$, we have
\begin{equation}\label{2.5}
\bE \left[\left(\Phi(Z_K^{(n)})-\Phi(K)\right)^k\right]  \le c_2 n^{-\frac{2k}{d+1}},
\end{equation}
where the constant $c_2$ is independent of $n$.
\end{theorem}

\begin{proof}
With $B_R$ as defined after Lemma \ref{L2.1}, we split
\begin{eqnarray*}
\bE \left[\left(\Phi(Z_K^{(n)})-\Phi(K)\right)^k\right] &=& \bE\left[{\mathbbm 1}\{Z_K^{(n)}\not\subset B_R\}\left(\Phi(Z_K^{(n)})-\Phi(K)\right)^k\right]\\
&&  +\,\bE\left[{\mathbbm 1}\{Z_K^{(n)}\subset B_R\}\left(\Phi(Z_K^{(n)})-\Phi(K)\right)^k\right].
\end{eqnarray*}
For the first summand we get
\begin{eqnarray*}
&&  \bE\left[{\mathbbm 1}\{Z_K^{(n)}\not\subset B_R\}\left(\Phi(Z_K^{(n)})-\Phi(K)\right)^k\right]\allowdisplaybreaks \\
&& \le \bE\left[{\mathbbm 1}\{R_o(Z_K^{(n)})>R\}R_o(Z_K^{(n)})^k\right] \allowdisplaybreaks\\
&& = \int_0^\infty \bP\left({\mathbbm 1}\{R_o(Z_K^{(n)})>R\}R_o(Z_K^{(n)})^k>t\right)\D t.
\end{eqnarray*}
We note that $ {\mathbbm 1}\{R_o(Z_K^{(n)})>R\}R_o(Z_K^{(n)})^k>t$ implies $R_o(Z_K^{(n)})>R$ and $R_o(Z_K^{(n)})^k>t$. Therefore, substituting $t=(R+x)^k$ for $t\ge R^k$ and using (\ref{2.1}), we get
\begin{eqnarray*}
&&  \bE\left[{\mathbbm 1}\{Z_K^{(n)}\not\subset B_R\}\left(\Phi(Z_K^{(n)})-\Phi(K)\right)^k\right]\allowdisplaybreaks \\
&& \le  \int_0^{R^k} \bP\left(R_o(Z_K^{(n)})>R\right)\D t +\int_{R^k}^\infty \bP\left( R_o(Z_K^{(n)})^k>t\right)\D t \allowdisplaybreaks\\
&& \le R^k 2de^{-cn} + \int_0^\infty \bP\left(R_o(Z_K^{(n)})>R+x\right)k(R+x)^{k-1}\D x\allowdisplaybreaks\\
&& \le  R^k 2d e^{-cn} +  2dk\int_0^\infty e^{-cn(x+1)}(R+x)^{k-1}\,\D x\allowdisplaybreaks\\
&& \le c_3 e^{-cn}.
\end{eqnarray*}
For the estimation of the second summand, we write $\varepsilon:= c(K)n^{-\frac{2}{d+1}}$ and apply Theorem \ref{T2.1}, to obtain
\begin{eqnarray*}
&& \bE\left[{\mathbbm 1}\{Z_K^{(n)}\subset B_R\}\left(\Phi(Z_K^{(n)})-\Phi(K)\right)^k\right] \\
&& = \int_0^\infty \bP\left({\mathbbm 1}\{Z_K^{(n)}\subset B_R\}(\Phi(Z_K^{(n)})-\Phi(K))^k>t\right)\D t\\
&& \le \varepsilon^k + \int_0^\infty \bP\left({\mathbbm 1}\{Z_K^{(n)}\subset B_R\}(\Phi(Z_K^{(n)})-\Phi(K))>\varepsilon+x\right)k(\varepsilon+x)^{k-1}\D x\\
&& \le \varepsilon^k + \int_0^\infty  c_1 e^{-nx}k(\varepsilon+x)^{k-1}\,\D x\\
&&\le \varepsilon^k + c_4\varepsilon^k.
\end{eqnarray*}
Both estimates together yield the estimate in (\ref{2.5}), first for sufficiently large $n$, but then by adaptation of the constant for all $n$.
\end{proof}

\section{The mean width difference}\label{sec3}

First in this section, we prove Theorem \ref{T1.1}. The upper bound is clear: if (\ref{1.4}) holds, then
$$ W(Z_K^{(n)})-W(K) \le \frac{2}{a_1}[\Phi(Z_K^{(n)})-\Phi(K)],$$
so that the upper estimate in (\ref{1.2}) follows from Theorem \ref{T1.2}.

It remains to prove the lower estimate. This is achieved by extending the proof in \cite[Section 4]{Sch19}. We use the approach of B\'{a}r\'{a}ny and Larman \cite{BL88}, in a dualized version. These authors consider a convex body $K\in \K$ with interior points and the convex hull, denoted by $K_n$, of $n$ independent uniform random points in $K$. For $x\in K$, they define
\begin{equation}\label{BL1}
v(x):= \min\{\lambda_d(K\cap H^+): x\in H^+,\,H^+ \mbox{ a closed halfspace}\},
\end{equation}
where $\lambda_d$ denotes Lebesgue measure in $\R^d$, and for $\varepsilon>0$,
\begin{equation}\label{BL2}
K(\varepsilon) := \{x\in K: v(x)\le\varepsilon\}.
\end{equation}
For sufficiently small $\varepsilon$, the closure of $K\setminus K(\varepsilon)$ was later called the `floating body' of $K$ with parameter $\varepsilon$, and $K(\varepsilon)$ the corresponding `wet part'. One result of B\'{a}r\'{a}ny and Larman \cite[Theorem 1]{BL88} says that
\begin{equation}\label{BL3}
V(K)- \bE V(K_n) \ge {\rm const}\cdot \lambda_d(K(1/n));
\end{equation}
and their Theorem 2 says that
\begin{equation}\label{BL4}
\lambda_d(K(\varepsilon)) \ge {\rm const}\cdot\varepsilon\log^{d-1}(1/\varepsilon)
\end{equation}
for sufficiently small $\varepsilon>0$. (The constant depends also on $K$, if their assumption $V(K)=1$ is deleted.)

First we argue in a dual way. The minimal volume $v(x)$ with $x\in K$ is replaced by a minimal measure of sets of hyperplanes determined by a hyperplane not intersecting $K$. For $y\in\R^d\setminus K$, let $K^y:= {\rm conv}(K\cup \{y\})$, and for hyperplanes $H\in\Ha\setminus \Ha_K$ define
$$ \psi(H):= \min\{\Phi(K^y)-\Phi(K):y\in H\}.$$
To show that the minimum exists, we recall that $o \in{\rm int}\,K$. Writing $y=\|y\|y_1$ with $y_1\in\Sd$, we have $ h(K^y,u) \ge h([o,y],u) = \|y\|\langle y_1,u\rangle^+$ for $u\in\Sd$, where $\langle\cdot\,,\cdot\rangle^+$ denotes the positive part. Therefore,
$$ \Phi(K) =\int_{\Sd} h(K^y,u)\,\varphi(\D u)\ge \|y\|\int_{\Sd} \langle y_1,u\rangle^+\, \varphi(\D u) \ge \|y\| b_0$$
with some constant $b_0>0$, by continuity and since the even measure $\varphi$ is not concentrated on a great subsphere. Now it is clear by continuity that the minimum is attained.

For $t\ge 0$, we define
$$ \Ha_K^\psi(t) := \{H\in \Ha\setminus \Ha_K: \psi(H)\le t\}.$$
Let $H\in\Ha\setminus \Ha_K$, and let $z\in H$ be such that $\psi(H)=\Phi(K^{z})-\Phi(K)$. If no hyperplane of $\widehat X$ separates $z$ and $K$, then $z\in Z_K^{(n)}$ and hence $H\cap Z_K^{(n)}\not=\emptyset$. It follows that
\begin{eqnarray*}
\bP\left(H\cap Z_K^{(n)}\not=\emptyset \right) &\ge& \bP\left(\widehat X\left(\Ha_{K^{z}}\setminus\Ha_K\right)=0\right)\\
&=&\exp\left[-\widehat\Theta(\Ha_{K^{z}}\setminus\Ha_K)\right]\\
&=&\exp\left[-2n(\Phi(K^z)-\Phi(K))\right]\\
&=& e^{-2n\psi(H)}.
\end{eqnarray*}
With the motion invariant measure
\begin{equation}\label{3.0}
\mu = \int_{\Sd} \int_{-\infty}^\infty {\mathbbm 1}\{H(u,\tau)\in\cdot\}\,\D\tau\,\sigma(\D u)
\end{equation}
we obtain
\begin{eqnarray*}
\bE W(Z_K^{(n)})-W(K) &=& \int_\Omega\int_{\Ha\setminus\Ha_K} {\mathbbm 1}\{H\cap Z_K^{(n)}\not=\emptyset\}\,\mu(\D H)\,\D\bP\allowdisplaybreaks\\
&=& \int_{\Ha\setminus\Ha_K} \bP(H\cap Z_K^{(n)}\not=\emptyset)\,\mu(\D H)\allowdisplaybreaks\\
&\ge& \int_{\Ha\setminus\Ha_K}e^{-2n\psi(H)}\,\mu(\D H)\allowdisplaybreaks\\
&\ge& \int_{\Ha\setminus \Ha_K} {\mathbbm 1}\{\psi(H)\le t\}e^{-2nt}\,\mu(\D H)\allowdisplaybreaks\\
&=& e^{-2nt}\mu(\Ha_K^\psi(t)),
\end{eqnarray*}
where $t> 0$ can be any number. The choice $t=1/n$ gives
\begin{equation}\label{3.1}
\bE W(Z_K^{(n)})-W(K) \ge  e^{-2}\mu(\Ha_K^\psi(1/n)).
\end{equation}
This serves as our dual counterpart to relation (\ref{BL3}).

In the next step, we carry over some results from \cite{BL88} to the dual body. Recalling that $K$ has its centroid at the origin, we let $K^\circ=\{x\in\R^d:\langle x,y\rangle\le 1\,\forall y\in K\}$ denote the polar body of $K$. For this body, we consider $K^\circ(t)$ for $t>0$, according to (\ref{BL2}). We choose $\rho,t_0>0$ such that $\rho B^d\subset K^\circ\setminus K^\circ(t)$ for $0<t\le t_0$.

Our aim is to show that, with suitable constants $a_2,a_3$ (depending only on $K,\varphi$) and for $t\ge 0$ sufficiently small,
\begin{equation}\label{BL5}
\mu(\Ha_K^\psi(a_2 t)) \ge a_3 \lambda_d(K^\circ(t)).
\end{equation}
If this is proved, then together with (\ref{3.1}) it yields
$$ \bE W(Z_K^{(n)})-W(K) \ge e^{-2}a_3\lambda_d(K^\circ((a_2 n)^{-1})),$$
if $n$ is sufficiently large. Now (\ref{BL4}), applied to the polar body and with suitable choices, yields the stated lower bound of Theorem \ref{T1.1}.

To prove (\ref{BL5}), we define the map $\eta: \R^d\setminus\{o\}\to \Ha\setminus\Ha_{\{o\}}$  by
$$ \eta(ru) = H(u,r^{-1})\quad\mbox{for } u\in \Sd,\,r> 0.$$
Let $\nu$ denote the image measure of the Lebesgue measure $\lambda_d$ under $\eta$, thus
\begin{equation}\label{3.3}
\nu(A) = \omega_d\int_{ \Sd}\int_0^\infty {\mathbbm 1}\{H(u,\tau)\in A\}\tau^{-(d+1)}\D\tau\,\sigma(\D u)
\end{equation}
for Borel sets $A$ of hyperplanes not passing through $o$.

Let $H(u,\tau)$ be a hyperplane contained in $\Ha_{\rho^{-1}B^d}\setminus \Ha_K$. Since $H(u,\tau)\cap K=\emptyset$, we have $\tau\ge h(K,u)$, which is bounded from below by a positive constant depending only on $K$. Since $H(u,\tau)\cap \rho^{-1}B^d\not=\emptyset$, we have $\tau\le \rho^{-1}$. By (\ref{3.3}) and (\ref{3.0}) there are constants $c_5,c_6>0$, depending only on $d$ and $K$, such that
$$ c_5\nu(A)\le \mu(A) \le c_6\nu(A)\quad \mbox{if } A\subset \Ha_{\rho^{-1}B^d}\setminus \Ha_K.$$

Let $0<t\le t_0$. Our aim is to prove an inclusion $\eta(K^\circ(t))\subset\Ha_K^{\psi}({\rm const}\cdot t)$. Let $x\in K^\circ(t)\setminus\partial K^\circ$. There is a hyperplane $E$ through $x$ that bounds a closed halfspace $E^+$ not containing $o$, such that $\lambda_d(K^\circ\cap E^+)\le t$. If $H:=\eta(x)$ and $y:=\eta^{-1}(E)$, then $y\in H\in \Ha_{\rho^{-1}B^d}\setminus \Ha_K$. We state that the mapping $\eta$ maps the cap $K^\circ\cap E^+$ bijectively onto the set of hyperplanes separating $y$ and $K$. For the proof, let $z\in K^\circ\cap E^+$ and write $z=ru$ with $u\in \Sd$  and $r>0$. If $r(K^\circ,\cdot)$ denotes the radial function of $K^\circ$, we have $r(K^\circ,u)\ge r$ and hence (see \cite[(1.52)]{Sch14}) $h(K,u)=1/r(K^\circ,u)\le 1/r$, thus $\eta(z)=H(u,1/r)\cap{\rm int}\,K=\emptyset$. Further, if $E=H(v,s)$ with $v\in \Sd$ and $s>0$, we have $\langle z,v\rangle \ge s$ and hence
$$ \langle y,u\rangle =\langle \eta^{-1}(E),u\rangle =\left\langle \frac{1}{s}v,u\right\rangle=\left\langle  \frac{1}{s}v, \frac{1}{r}z\right\rangle \ge  \frac{1}{r}.$$
Hence, the hyperplane $\eta(z)= H(u,1/r)$ separates $K$ and $z$. The arguments can be reversed, which completes the proof of the statement. We denote the set of hyperplanes separating $y$ and $K$ by $\Ha^y_K$. Below it is used that each hyperplane from $\Ha^y_K$ meets the ball $\rho^{-1}B^d$. It follows that
$$ W(K^y)-W(K)=\mu(\Ha_K^y) \le c_6\nu(\Ha_K^y) = c_6\lambda_d(K^\circ\cap E^+) \le c_6t.$$

Now we use the assumption (\ref{1.3}). It yields that
$$ \Phi(K^y)-\Phi(K) =\int_{ \Sd}[h(K^y,u)-h(K,u)]\,\varphi(\D u) \le c_7[W(K^y)-W(K)]$$
with $c_7=a_0/2$. Therefore, $\psi(H)\le c_7c_6t$, which gives $H\in \Ha_K^\psi(c_7c_6t)$. Since  $x\in K^\circ(t)\setminus\partial K^\circ$ was arbitrary, this shows that $\eta(K^\circ(t)) \subset \Ha_K^\psi(c_7c_6t)$. Therefore,
$$ \lambda_d(K^\circ(t)) =\nu(\eta(K^\circ(t)))\le \nu(\Ha_K^\psi(c_7c_6t)) \le c_5^{-1}\mu(\Ha_K^\psi(c_7c_6t)).$$
This is the stated inequality (\ref{BL5}) and thus completes the proof of Theorem \ref{T1.1}. \hfill$\Box$

\vspace{2mm}

That the orders in Theorem \ref{T1.1} are optimal, follows from exact asymptotic relations, which will now be proved. They are of independent interest.

We  define
$$ K_1:= K+B^d.$$
This is the outer parallel body of $K$ at distance $1$. The next two lemmas serve to control the error that we make when we restrict ourselves to $K$-cells contained in $K_1$. This requires an assumption on the directional distribution
$\varphi$, which ensures that every convex body $K$ can be approximated by the corresponding $K$-cells.

\begin{lemma}\label{L3.1}
If the support of $\varphi$ is all of $\Sd$, then
\begin{equation}\label{3.4}
\bP\left(Z_K^{(n)}\not\subset K_1\right)\le c_8e^{-a_4n} ,
\end{equation}
with constants $a_4,c_8$ depending only on $K,\varphi$.
\end{lemma}

\begin{proof}
Suppose that ${\rm supp}\,\varphi=\Sd$. There are a number $m\in\N$ and  unit vectors $v_1,\dots,v_m\in\Sd$ such that
$$ \bigcap_{i=1}^m H^-(K,v_i)\subset K+\frac{1}{2}B^d,$$
where $H^-(K,v)$ denotes the supporting halfspace of $K$ with outer unit normal vector $v$. We can choose pairwise disjoint neighborhoods $U_i\subset \Sd$ of $v_i$, for $i=1,\dots,m$ (depending only on $K$) such that
$$ \bigcap_{i=1}^m H^-(K,u_i) \subset K_1 \quad \mbox{whenever }u_i\in U_i \mbox{ for }i=1,\dots,m.$$
We set $\varpi:= \min\{\varphi(U_i):i=1,\dots,m\}$. By the assumption on $\varphi$, the number $\varpi$ is positive; further, it depends only on $\varphi,K$. Now we argue similarly as in the proof of Lemma 1 in \cite{Sch19}. The sets of hyperplanes
$$ A_i:= \{H(u,\tau): u\in U_i,\, h(K,u)\le \tau\le h(K,u)+1\},\quad i=1,\dots,m,$$
are pairwise disjoint. If $\widehat X(A_i) >0$ for $i=1,\dots,m$, then $Z_K^{(n)}\subset K_1$. Therefore, observing that $\widehat\Theta(A_i) = n\varphi(U_i)$ for $i=1,\dots,m $, we get
\begin{eqnarray*}
\bP(Z_K^{(n)}\not\subset K_1) &\le& \bP( \widehat X(A_i)=0\mbox{ for at least one } i\in \{1,\dots,m\})\allowdisplaybreaks\\
&=& 1-\prod_{i=1}^{m} (1-\bP( \widehat X(A_i))=0))\allowdisplaybreaks\\
&=& 1-\prod_{i=1}^{m} (1- e^{-n\varphi(U_i)})\allowdisplaybreaks\\
&\le&  1-(1-e^{- n\varpi})^m\allowdisplaybreaks\\
&\le& me^{-n\varpi},
\end{eqnarray*}
which proves the assertion.
\end{proof}

\begin{lemma}\label{L3.2}
Suppose that ${\rm supp}\,\varphi=\Sd$. Then
$$ \bE\left[W(Z_K^{(n)}){\mathbbm 1}\{Z_K^{(n)}\not\subset K_1\}\right] \le c_9e^{-a_5 n}$$
with constants $a_5,c_9$ depending only on $K,\varphi$.
\end{lemma}

\begin{proof}
The Cauchy--Schwarz inequality yields
$$ \bE\left[W(Z_K^{(n)}){\mathbbm 1}\{Z_K^{(n)}\not\subset K_1\}\right]\le \left(\bE\left[W(Z_K^{(n)})^2\right]\right)^{\frac{1}{2}}\bP\left(Z_K^{(n)}\not\subset K_1\right)^{\frac{1}{2}}.$$
The moments of $W(Z_K^{(n)})$ are bounded, as follows from
$$ \bE\left[W(Z_K^{(n)})^k\right] \le \bE\left[2^kR_o(Z_K^{(n)})^k\right]=2^k \int_0^\infty \bP\left(R_o(Z_K^{(n)})^k>t\right)\D t$$
for $k\in\N$. The last integral can be treated as in the proof of Theorem \ref{T2.2}. Therefore, the assertion follows from Lemmas \ref{L2.1} and \ref{L3.1}.
\end{proof}

The following lemma allows us to carry over to Poisson processes certain asymptotic relations for expectations that hold for finitely many i.i.d. points or hyperplanes. The lemma was first used in a special case by Reitzner \cite[Lemma 1]{Rei05}. We found it necessary to provide more details of the proof.

\begin{lemma}\label{L3.3}
Let $f:\R^+\to\R^+$ be one of the functions

\noindent {\em Case (a):} $f(\lambda)=\lambda^\alpha$ with $0<\alpha<1$,

\noindent {\em Case (b):} $f(\lambda)=\lambda\log^{-a}\lambda$ with $a>0$.

\noindent Let $(g_k)_{k\in \N}$ be a sequence with
$$ \lim_{k\to\infty} f(k) g_k=g<\infty.$$
Let $N_\lambda$ be a real random variable which has a Poisson distribution with parameter $\lambda>0$. Then
\begin{equation}\label{A1}
\lim_{\lambda\to\infty} f(\lambda)\sum_{k=0}^\infty \bP(N_\lambda=k) g_k=g.
\end{equation}
\end{lemma}

\begin{proof}
We define the function $h:\R^+\to\R^+$ as follows:

\noindent {\em Case} (a): $h(\lambda)=\lambda^\beta$ with $(1+\alpha)/2<\beta<1$,

\noindent {\em Case} (b): $h(\lambda)=\lambda\log^{-b}\lambda$ with $0<b<a/2$.

\noindent Then
\begin{equation}\label{3.5}
\frac{\lambda f(\lambda)}{h(\lambda)^2}\to 0 \quad\mbox{and}\quad \frac{h(\lambda)}{\lambda}\to 0\quad\mbox{as }\lambda\to\infty.
\end{equation}
Define
\begin{eqnarray*}
A(\lambda) &:=& \sum_{|k-\lambda|>h(\lambda)} \bP(N_\lambda=k)\left[f(\lambda) g_k-g\right],\\
B(\lambda) &:=& \sum_{|k-\lambda|\le h(\lambda)} \bP(N_\lambda=k)\left[f(\lambda) g_k-g\right]
\end{eqnarray*}
($k\in\N_0$ in each case), so that
$$ A(\lambda)+B(\lambda)=\left(\sum_{k=0}^\infty \bP(N_\lambda=k)f(\lambda) g_k\right)-g.$$

Concerning the first sum, we note that the sequence $(g_k)_{k\in \N}$ is bounded, hence there is a constant $c_{10}$ with
$$|f(\lambda) g_k-g|\le f(\lambda) c_{10} \quad\mbox{for all $k$},$$
for sufficiently large $\lambda$. Using this and Tschebyscheff's inequality, we get
\begin{eqnarray*}
|A(\lambda)| &\le& \sum_{|k-\lambda|>h(\lambda)} \bP(N_\lambda=k)f(\lambda) c_{10}\\
&=& \bP(|N_\lambda-\lambda|>h(\lambda))f(\lambda) c_{10}\\
&\le& \frac{\lambda}{h(\lambda)^2}f(\lambda) c_{10}.
\end{eqnarray*}
By (\ref{3.5}), this tends to zero as $\lambda\to\infty$.

For the second sum, we obtain
\begin{eqnarray*}
|B(\lambda)| &\le& \sum_{|k-\lambda|\le h(\lambda)} \bP(N_\lambda =k)\Bigg|\frac{f(\lambda)}{f(k)}f(k)g_k-g\Bigg|\\
&\le& \sum_{|k-\lambda|\le h(\lambda)} \bP(N_\lambda =k)\left(\Bigg|\frac{f(\lambda)}{f(k)}-1\Bigg| |f(k) g_k|+|f(k) g_k-g|\right).
\end{eqnarray*}

Let $\varepsilon>0$ be given. There is some number $k_0$ with $|f(k) g_k-g|<\varepsilon$ for $k\ge k_0$. Further, there is a number $\lambda_0$ with $\lambda-h(\lambda) \ge k_0$ for $\lambda\ge \lambda_0$. Hence, for $\lambda \ge \lambda_0$ we have
\begin{eqnarray*}
|B(\lambda)| &\le& \sum_{|k-\lambda|\le h(\lambda)} \bP(N_\lambda =k)\left(\Bigg|\frac{f(\lambda)}{f(k)}-1\Bigg| |f(k) g_k|+\varepsilon\right)\\
&\le& \varepsilon + c_{11}\sum_{|k-\lambda|\le h(\lambda)} \bP(N_\lambda =k)\Bigg|\frac{f(\lambda)}{f(k)}-1\Bigg|
\end{eqnarray*}
with a constant $c_{11}$, since the sequence $(f(k) g_k)_{k\in\N}$ is bounded.

Suppose that $|k-\lambda|\le h(\lambda)$. Then $0<\lambda-h(\lambda) \le k\le \lambda+h(\lambda)$ and hence
\begin{equation}\label{3.6}
\Big|\frac{\lambda}{k}-1\Big| \le \frac{h(\lambda)}{k} \le \frac{h(\lambda)}{\lambda-h(\lambda)} =\frac{1}{\frac{\lambda}{h(\lambda)}-1} \to 0\quad\mbox{as }\lambda\to\infty,
\end{equation}
by (\ref{3.5}). The function $f$ has the following property:
$$ \forall\varepsilon>0\;\exists\delta>0\;\exists x_0>0\;\forall x\ge x_0: |\gamma-1|<\delta\Rightarrow \Bigg|\frac{f(\gamma x)}{f(x)}-1\Bigg|<\varepsilon.$$
From this and (\ref{3.6}) it follows that $B(\lambda)\to 0$ as $\lambda\to\infty$. Both limit results together yield the assertion.
\end{proof}

Now we prove an exact asymptotic relation, by using a result of B\"or\"oczky, Fodor and Hug \cite[Thm. 5.2]{BFH10}.

In the following, we shall assume that the directional distribution $\varphi$ of $\widehat X$ has a positive, continuous density $q$ with respect to spherical Lebesgue measure. We need the following functional depending on the convex body $K$ and on $q$:
$$ F(K,q) = 2c_d\omega_d^{-\frac{d-1}{d+1}} \int_{\partial K} q(\sigma_K(x))^{-\frac{2}{d+1}}\kappa(x)^\frac{d}{d+1}\,{\mathscr H}^{d-1}(\D x),$$
where ${\mathscr H}^{d-1}$ denotes the $(d-1)$-dimensional Hausdorff measure. The constant $c_d$ is defined by
$$ c_d=\frac{(d^2+d+2)(d^2+1)}{2(d+3)\cdot(d+1)!}\Gamma\left(\frac{d^2+1}{d+1}\right)\left(\frac{d+1}{\kappa_{d-1}}\right)^{\frac{2}{d+1}},$$
and $\sigma_K(x)$ is the ${\mathscr H}^{d-1}$-almost everywhere unique outer unit normal vector of $K$ at the point $x\in\partial K$. The Gauss--Kronecker curvature $\kappa$ of $\partial K$ exists ${\mathscr H}^{d-1}$-almost everywhere on $\partial K$. We have $F(K,q)>0$ if $K$ is of class $C^2$.

\begin{theorem}\label{T3.1}
Suppose that the directional distribution $\varphi$ of $\widehat X$ has a positive continuous density $q$. Then
$$ \lim_{n\to\infty} n^{\frac{2}{d+1}}\left[\bE W(Z_K^{(n)})-W(K)\right] = 2^{-\frac{2}{d+1}}F(K,q).$$
\end{theorem}

\begin{proof}
First we observe that
$$ W(Z_K^{(n)}) = W(Z_K^{(n)}\cap K_1) +\left(W(Z_K^{(n)}) - W(Z_K^{(n)}\cap K_1)\right){\mathbbm 1}\{Z_K^{(n)}\not\subset K_1\}.$$
In the following, we abbreviate
$$ \Ha_{K_1}^*:= \Ha_{K_1} \setminus \Ha_K \quad \mbox{and} \quad X_K^{(n)}:= \widehat X\fed \Ha_{K_1}^*.$$
For any $k\in\N_0$, we get
\begin{eqnarray*}
&& \bE\left[W(Z_K^{(n)}){\mathbbm 1}\{X_K^{(n)}(K_1)=k\}\right]\\
&&= \bE\left[W(Z_K^{(n)}\cap K_1){\mathbbm 1}\{X_K^{(n)}(K_1)=k\}\right] + O(e^{-a_5 n})
\end{eqnarray*}
by Lemma \ref{L3.2}, where the constant involved in $O$ is independent of $k$. Thus, we obtain the expansion
\begin{eqnarray*}
&& \bE W(Z_K^{(n)})\\
&& = \sum_{k=0}^\infty \bP\left(X_K^{(n)}(K_1)=k\right)\bE\left[W(Z_K^{(n)}\cap K_1)\mid X_K^{(n)}(K_1)=k\right] + O(e^{-a_5 n}).
\end{eqnarray*}
We define a Borel measure $\mu_K$ on $\Ha$ by
\begin{equation}\label{3.7}
\mu_K:= \int_{\Sd}\int_{h(K,u)}^{h(K,u)+1} {\mathbbm 1}\{H(u,t)\in\cdot\}\,\D t\,\varphi(\D u).
\end{equation}
In the following, $\Theta_n$ denotes the intensity measure of $X_K^{(n)}$; then $\Theta_n=2n\mu_K$. The measure $\mu_K$ is a probability measure on $\Ha$ which is concentrated on $\Ha_{K_1}^*$, and $\Theta_n(\Ha_{K_1}^*)=2n$. Thus we get
\begin{eqnarray*}
&& \bE W(Z_K^{(n)})+ O(e^{-a_5 n})\\
&& = \sum_{k=0}^\infty e^{-\Theta_n(\Ha_{K_1}^*)}\frac{\Theta_n(\Ha_{K_1}^*)^k}{k!} \int_{(\Ha_{K_1}^*)^k} W\left(\bigcap_{i=1}^k H^-_i\cap K_1\right)\mu_K^k(\D(H_1,\dots,H_k)).
\end{eqnarray*}
Here for $H\in\Ha_{K_1}^*$ we denote by $H^-$ the closed halfspace bounded by $H$ that contains $K$.

The random variable $X_K^{(n)}(K_1)$ has a Poisson distribution with parameter $2n$. In the following, we denote such a random variable by $N_{2n}$. Denoting by $h_1,\dots,h_k$ i.i.d. random hyperplanes with distribution $\mu_K$, it follows from the preceding argumentation that
$$ \bE W(Z_K^{(n)})-W(K)=\sum_{k=0}^\infty \bP(N_{2n}=k)g_k+O(e^{-a_5 n})$$
with
$$ g_k:= \bE\left[W\left(\bigcap_{i=1}^k h_i^-\cap K_1\right)\right]-W(K).$$
It was shown in \cite[Thm. 5.2]{BFH10} that
$$ \lim_{k\to\infty} k^{\frac{2}{d+1}}g_k=F(K,q).$$
Now Lemma \ref{L3.3} (Case (a) and with $\lambda= 2n$) gives the assertion.
\end{proof}

In a similar way, we can derive from Theorem 1.3 in B\"or\"oczky and Schneider \cite{BS10} the following theorem. Here Case (b) of Lemma \ref{L3.3} is needed.

\begin{theorem}\label{T3.2}
Suppose that $\widehat X$ is isotropic. If $K$ is a simplicial polytope with $r$ facets, then
$$ \lim_{n\to\infty}  \frac{n}{\log^{d-1} n} \left[\bE W(Z_K^{(n)})-W(K)\right] = rd\left(\frac{\log 2}{d+1}\right)^{d-1}.$$
\end{theorem}

\section{The facet number}\label{sec4}

The hitting number difference is related to the facet number of the $K$-cell. Denoting the number of facets of a polytope $P$ by $f_{d-1}(P)$, we have
\begin{equation}\label{4.1}
\bE f_{d-1}(Z_K^{(n)}) = 2n\,\bE[\Phi(Z_K^{(n)})-\Phi(K)].
\end{equation}
For the proof, we note that for a hyperplane $H\in\Ha\setminus\Ha_K$ we have
$${\mathbbm 1}\{Z_K(\widehat X)\cap H\not=\emptyset\} = {\mathbbm 1}\{Z_K(\widehat X\cup\{H\})\cap H\not=\emptyset\},$$
where $Z_K({\sf H})$ denotes the $K$-cell generated by the hyperplane system ${\sf H}$. Therefore, the Slivnyak--Mecke formula (Cor. 3.2.3 in \cite{SW08}) yields
\begin{eqnarray*}
\bE f_{d-1}(Z_K^{(n)}) &=& \bE \sum_{H\in\widehat X} {\mathbbm 1}\{Z_K^{(n)}\cap H\not=\emptyset\}{\mathbbm 1}\{H\cap K=\emptyset\}\\
&=& \int_\Ha \bE{\mathbbm 1}\{Z_K^{(n)}\cap H\not=\emptyset\}{\mathbbm 1}\{H\cap K=\emptyset\}\,\widehat\Theta(\D H)\\
&=& 2n \int_{\Sd} \bE \int_{h(K,u)}^\infty {\mathbbm 1}\{Z_K^{(n)}\cap H(u,\tau)\not=\emptyset\}\,\D\tau\,\varphi(\D u)\\
&=& 2n \int_{\Sd} \bE[h(Z_K^{(n)},u)-h(K,u)]\,\varphi(\D u)\\
&=& 2n\bE[\Phi(Z_K^{(n)})-\Phi(K)].
\end{eqnarray*}
Formula (\ref{4.1}) can be seen as a dual and Poisson counterpart to the Efron  identity (formula (8.12) in \cite{SW08}). The following lemma is, in a similar way, a counterpart to Lemma 5 in Brunel \cite{Bru17}.

\begin{lemma}\label{L4.1}
Let $k\in\N$. Abbreviating $f_{d-1}(Z_K^{(n)})=:f$, we have
$$ \bE\left[f(f-1)\cdots(f-k+1)\right] \le \frac{(2n)^k}{k!}\,\bE\left[\left(\Phi(Z_K^{(n)})-\Phi(K)\right)^k\right].$$
\end{lemma}

\begin{proof}
For a $d$-polytope $P$, we denote by ${\sf F}(P)$ the set of its facet hyperplanes, that is, the affine hulls of its facets. For $k\in\N$,
$$ \frac{1}{k!}\sum_{(H_1,\dots,H_k)\in\widehat X^k_{\not=}} \prod_{i=1}^k {\mathbbm 1}\left\{H_i\in{\sf F}(Z_K^{(n)})\right\}$$
is the number of ordered $k$-tuples of facets of $Z_K^{(n)}$ and hence is equal to $f(f-1)\cdots(f-k+1)$. Therefore, the Slivnyak--Mecke formula  yields
\begin{eqnarray*}
&& k! \bE [f(f-1)\cdots(f-k+1)]\\
&& = \bE \sum_{(H_1,\dots,H_k)\in\widehat X^k_{\not=}} \prod_{i=1}^k {\mathbbm 1}\left\{H_i\in{\sf F}(Z_K^{(n)})\right\} {\mathbbm 1}\{H_i\cap K=\emptyset\}\allowdisplaybreaks\\
&& = \int_{\Ha^k} \bE \prod_{i=1}^k {\mathbbm 1}\left\{ H_i\in{\sf F}(Z_K(\widehat X\cup\{H_1,\dots,H_k\}))\right\}{\mathbbm 1}\{H_i\cap K=\emptyset\}\,\widehat\Theta^k(\D(H_1,\dots,H_k))\allowdisplaybreaks\\
&& \le \int_{\Ha^k} \bE \prod_{i=1}^k {\mathbbm 1}\{H_i\cap Z_K^{(n)}\not=\emptyset\}{\mathbbm 1}\{H_i\cap K=\emptyset\}\,\widehat\Theta^k(\D(H_1,\dots,H_k))\allowdisplaybreaks\\
&& =n^k\, \bE \int_{(\Sd)^k} \int_{\R^k} \prod_{i=1}^k{\mathbbm 1}\{H(u_i,\tau_i)\cap Z_K^{(n)}\not=\emptyset\}{\mathbbm 1}\{H(u_i,\tau_i)\cap K=\emptyset\}\allowdisplaybreaks\\
&& \hspace{4mm}\times\;\D(\tau_1,\dots,\tau_k)\,\varphi^k(\D(u_1,\dots,u_k))\allowdisplaybreaks\\
&& = n^k\, \bE   \prod_{i=1}^k 2 \int_{(\Sd)^k}[h(Z_K^{(n)},u_i)-h(K,u_i)]\,\varphi(\D u_i)\allowdisplaybreaks\\
&& = (2n)^k \,\bE \left[\left(\Phi(Z_K^{(n)})-\Phi(K)\right)^k\right]
\end{eqnarray*}
and thus the assertion.
\end{proof}

Since the polynomial $x^k$ is a linear combination of the polynomials $x(x-1)\cdots(x-j+1)$, $0\le j\le k$, we obtain the inequality
\begin{equation}\label{4.5}
\bE \left[f_{d-1}(Z_K^{(n)})^k\right] \le C_k(2n)^k\, \bE\left[\left(\Phi(Z_K^{(n)})-\Phi(K)\right)^k\right],
\end{equation}
where $C_k$ is a constant depending only on $d$ and $k$.

Together with Theorem \ref{T2.2}, this yields the following.

\begin{theorem}\label{T4.1}
For $k\in \N$,
\begin{equation}\label{4.1a}
\bE \left[f_{d-1}(Z_K^{(n)})^k\right] \le c(k)n^{\frac{k(d-1)}{d+1}}
\end{equation}
with a constant $c(k)$ independent of $n$.
\end{theorem}

Now we prove an exact asymptotic relation, which shows that the order in (\ref{4.1a}) for $k=1$ is optimal. We derive this from another result of B\"or\"oczky, Fodor and  Hug \cite{BFH10}. For $k\in\N_0$, let $h_1,\dots,h_k$ be i.i.d. random hyperplanes with distribution $\mu_K$ given by (\ref{3.7}). Define
$$ p_k:= \bE\left[f_{d-1}\left(\bigcap_{i=1}^k h_i^-\right){\mathbbm 1}\left\{\bigcap_{i=1}^k h_i^-\subset K_1\right\}\right],$$
where $h_i^-$ is the closed halfspace bounded by $h_i$ that contains $K$. Suppose that $\varphi$ has a positive, continuous density $q$, and define
$$ G(K,q):= c_d(\omega_d)^{-\frac{d-1}{d+1}}\int_{\partial K} q(\sigma_K(x))^{\frac{d-1}{d+1}}\kappa(x)^{\frac{d}{d+1}}{\mathscr H}^{d-1}(\D x).$$
It was shown in \cite[Thm. 5.3]{BFH10} that
$$ \lim_{k\to\infty} k^{-\frac{d-1}{d+1}} p_k = p:= G(K,q).$$
In fact, the formulation in \cite{BFH10} is slightly different, but it is explained there on page 502 that the indicator ${\mathbbm 1}\{\cdot\subset K_1\}$ may be inserted without changing the limit relation. Again we have $G(K,q)>0$ if $K$ is of class $C^2$.

We provide two more auxiliary results. They are motivated by a remark of Calka and Schreiber \cite[p. 48]{CS06} (in a dual situation). We found it necessary to give details of the proof.

\begin{lemma}\label{concentration}
Let $N_\lambda$ have a Poisson distribution with parameter $\lambda>0$. Let $\frac{1}{2}<b<a<1$. There are constants $c_{12}, c_{13}>0$ such that
$$ \bP\left(N_\lambda\ge \lambda-1+\lambda^a\;\mbox{or}\; N_\lambda\le \lambda-1-\lambda^a\right) \le c_{12}\,\exp(-c_{13}\lambda^{2b-1}).$$
\end{lemma}

\begin{proof}
If $x>\lambda$, then
$$ \bP(N_\lambda\ge x) \le \exp\left[ x-\lambda-x\log\left(\frac{x}{\lambda}\right)\right],$$
and if $x<\lambda$, then
$$ \bP(N_\lambda\le x) \le \exp\left[ x-\lambda-x\log\left(\frac{x}{\lambda}\right)\right],$$
see \cite[Thm 5.4]{Mitzen}.

Let $\lambda$ be so large that $-1+\lambda^a\ge \lambda^b$ and $-1-\lambda^a\le -\lambda^b$. Then
\begin{equation}\label{MU1}
\bP\left(N_\lambda\ge \lambda-1+\lambda^a\right) \le \bP\left(N_\lambda\ge \lambda+\lambda^b\right)
\le \exp\left[\lambda^b-(\lambda+\lambda^b)\log(1+\lambda^{b-1})\right].
\end{equation}
Since $\log(1+z)\ge z-\frac{1}{2}z^2$ for $z\in[0,1]$, expression (\ref{MU1}) is $\le \exp[-c_{14} \lambda^{2b-1}$] with a constant $c_{14}>0$, for sufficiently large $\lambda$. Further,
\begin{equation}\label{MU2}
\bP\left(N_\lambda\le \lambda-1-\lambda^a\right) \le \bP\left(N_\lambda\le \lambda-\lambda^b\right)
\le \exp\left[-\lambda^b-(\lambda-\lambda^b)\log(1-\lambda^{b-1})\right].
\end{equation}
Since $\log(1-z)\ge -z-\frac{4}{5}z^2$ for $z\in[0,1/2]$, expression (\ref{MU2}) is $\le \exp[-c_{15} \lambda^{2b-1}$] with a constant $c_{15}>0$, for sufficiently large $\lambda$. By adapting the constants, we obtain the assertion of the lemma for all $\lambda>0$.
\end{proof}

The following auxiliary result complements Lemma \ref{L3.3}.

\begin{lemma}\label{lambdaasymptotics}
Let $\bar{f}(\lambda):=\lambda^{-\beta}$ for $\lambda>0$, where $0<\beta<1$, and suppose that $\bar{f}(k)p_k\to p<\infty$ as $k\to\infty$.
Then
$$
\lim_{\lambda\to\infty}\bar{f}(\lambda)\sum_{k=0}^\infty\bP(N_\lambda=k)p_k=p.
$$
\end{lemma}

\begin{proof} We proceed similarly as in the proof of Lemma \ref{L3.3}. We choose $\frac{1}{2}<b<a<1$ and set $h(\lambda):=\lambda^a$. Then we define
\begin{eqnarray*}
\bar{A}(\lambda) &:=& \sum_{|k-\lambda|>h(\lambda)} \bP(N_\lambda=k)\left[\bar{f}(\lambda) p_k-p\right],\\
\bar{B}(\lambda) &:=& \sum_{|k-\lambda|\le h(\lambda)} \bP(N_\lambda=k)\left[\bar{f}(\lambda) p_k-p\right]
\end{eqnarray*}
($k\in\N_0$ in each case), so that
$$ \bar{A}(\lambda)+\bar{B}(\lambda)=\left(\sum_{k=0}^\infty \bP(N_\lambda=k)\bar{f}(\lambda) p_k\right)-p.$$
It can be shown exactly as in the proof of Lemma \ref{L3.3} that $\bar{B}(\lambda)\to 0$ as $\lambda\to\infty$. In order to see that also $\bar{A}(\lambda)\to 0$ as $\lambda\to\infty$, we first observe that for $\lambda\ge 1$ we have
$$ |\bar{f}(\lambda)p_k-p|\le c_{16} \bar{f}(k)^{-1}\le c_{16}k,$$
with a constant $c_{16}>0$ independent of $k\in\N_0$ and $\lambda\ge 1$. Hence, if $\lambda\ge 1$, then
\begin{align*}
|\bar{A}(\lambda)|&\le c_{16}\sum_{|k-\lambda|>h(\lambda)} k\, \bP(N_\lambda=k)\\
&=c_{16}\lambda \sum_{|k-\lambda|>h(\lambda),\,k\ge 1} e^{-\lambda}\frac{\lambda^{k-1}}{(k-1)!}\\
&=c_{16}\lambda \bP\left(N_\lambda> \lambda-1+h(\lambda)\text{ or } N_\lambda< \lambda-1-h(\lambda)\right)\\
&\le c_{16}\lambda c_{12}\exp(-c_{13}\lambda^{2b-1}),
\end{align*}
which tends to zero as $\lambda\to\infty$.
\end{proof}

After these preparations, we can show the following.

\begin{theorem}\label{T4.2}
Suppose that the directional distribution $\varphi$ of $\widehat X$ has a positive, continuous density $q$. Then
\begin{equation}\label{4.2}
\lim_{n\to\infty} n^{-\frac{d-1}{d+1}}\bE f_{d-1}(Z_K^{(n)}) = 2^{\frac{d-1}{d+1}}G(K,q).
\end{equation}
\end{theorem}

\begin{proof}
First we note that the Cauchy--Schwarz inequality yields
$$ \bE\left[f_{d-1}(Z_K^{(n)}){\mathbbm 1}\{Z_K^{(n)}\not\subset K_1\}\right]
\le \left(\bE\left[f_{d-1}(Z_K^{(n)})^2\right]\right)^{\frac{1}{2}}\bP\left(Z_K^{(n)}\not\subset K_1\right)^{\frac{1}{2}},$$
which tends to zero as $n\to\infty$, by Theorem \ref{T4.1} and Lemma \ref{L3.1}. Hence, with $\bar{f}(\lambda):=\lambda^{-\frac{d-1}{d+1}}$,
\begin{align*}
&\lim_{n\to\infty}\left(\bar{f}(n)\bE f_{d-1}(Z_K^{(n)})\right)\\
&=\lim_{n\to\infty}\left(\bar{f}(n)\bE\left[ f_{d-1}(Z_K^{(n)}){\mathbbm 1}\{Z_K^{(n)}\subset K_1\}\right]\right)\\
&=\lim_{n\to\infty}\left(\bar{f}(n)\sum_{k=0}^\infty\bP(X^{(n)}_K(K_1)=k)\bE\left[ f_{d-1}
(Z_K^{(n)}){\mathbbm 1}\{Z_K^{(n)}\subset K_1\}\mid X^{(n)}_K(K_1)=k \right]\right)\\
&=\lim_{n\to\infty}\left(\bar{f}(n)\sum_{k=0}^\infty\bP(N_{2n}=k)p_k\right)\\
&=2^{\frac{d-1}{d+1}}\lim_{n\to\infty}\left(\bar{f}(2n)\sum_{k=0}^\infty\bP(N_{2n}=k)p_k\right)\\
&=2^{\frac{d-1}{d+1}}p=2^{\frac{d-1}{d+1}}G(K,q),
\end{align*}
where Lemma \ref{lambdaasymptotics} was used in the last step.
\end{proof}

In the isotropic case, we obtain from Theorem \ref{T3.2} and equality (\ref{4.1}) the following asymptotic relation.

\begin{theorem}\label{T4.3}
Suppose that $\widehat X$ is isotropic.  If $K$ is a simplicial polytope with $r$ facets, then
\begin{equation}\label{4.3}
\lim_{n\to\infty} (\log^{1-d}n)\, \bE f_{d-1}(Z_K^{(n)}) = rd\left(\frac{\log 2}{d+1}\right)^{d-1}.
\end{equation}
\end{theorem}

\noindent Authors's addresses:\\[2mm]
Daniel Hug\\Karlsruhe Institute of Technology, Department of Mathematics\\D-76128 Karlsruhe, Germany\\E-mail: daniel.hug@kit.edu\\[2mm]
Rolf Schneider\\Mathematisches Institut, Albert-Ludwigs-Universit{\"a}t\\D-79104 Freiburg i. Br., Germany\\E-mail: rolf.schneider@math.uni-freiburg.de

\end{document}